\documentclass{amsart}
\usepackage[T1]{fontenc}
\usepackage[latin9]{inputenc}
\usepackage[english]{babel}
\usepackage{float}
\usepackage{textcomp}
\usepackage{amsthm}
\usepackage{amsmath}
\usepackage{amssymb}
\usepackage{graphicx}
\usepackage[unicode=true,
 bookmarks=true,bookmarksnumbered=true,bookmarksopen=false,
 breaklinks=false,pdfborder={0 0 1},backref=false,colorlinks=false]
 {hyperref}
\hypersetup{pdftitle={Exotic Cluster Structures on SL5},
 pdfauthor={Idan Eisner}}
\usepackage{breakurl}

\makeatletter

\newcommand{\lyxmathsym}[1]{\ifmmode\begingroup\def\b@ld{bold}
  \text{\ifx\math@version\b@ld\bfseries\fi#1}\endgroup\else#1\fi}

\providecommand{\tabularnewline}{\\}

\theoremstyle{plain}
\newtheorem{thm}{\protect\theoremname}
  \theoremstyle{plain}
  \newtheorem{prop}[thm]{\protect\propositionname}
  \theoremstyle{plain}
  \newtheorem{conjecture}[thm]{\protect\conjecturename}
  \theoremstyle{definition}
  \newtheorem{example}[thm]{\protect\examplename}
  \theoremstyle{remark}
  \newtheorem{rem}[thm]{\protect\remarkname}

\DeclareMathOperator{\diag}{diag}
\DeclareMathOperator{\rank}{rank}
\DeclareMathOperator{\End}{End}
\DeclareMathOperator{\spn}{span}

\makeatother

  \addto\captionsamerican{\renewcommand{\conjecturename}{Conjecture}}
  \addto\captionsamerican{\renewcommand{\examplename}{Example}}
  \addto\captionsamerican{\renewcommand{\propositionname}{Proposition}}
  \addto\captionsamerican{\renewcommand{\remarkname}{Remark}}
  \addto\captionsamerican{\renewcommand{\theoremname}{Theorem}}
  \addto\captionsenglish{\renewcommand{\conjecturename}{Conjecture}}
  \addto\captionsenglish{\renewcommand{\examplename}{Example}}
  \addto\captionsenglish{\renewcommand{\propositionname}{Proposition}}
  \addto\captionsenglish{\renewcommand{\remarkname}{Remark}}
  \addto\captionsenglish{\renewcommand{\theoremname}{Theorem}}
  \providecommand{\conjecturename}{Conjecture}
  \providecommand{\examplename}{Example}
  \providecommand{\propositionname}{Proposition}
  \providecommand{\remarkname}{Remark}
\providecommand{\theoremname}{Theorem}

\begin{document}

\title{Exotic cluster structures on $SL_{5}$}

\date{}
\author{Idan Eisner}
\address{Department of Mathematics, University of Haifa, Haifa,
Mount Carmel 31905, Israel}
\email{eisner@math.haifa.ac.il}

\begin{abstract}
A conjecture by Gekhtman, Shapiro and Vainshtein suggests a correspondence
between the Belavin--Drinfeld classification of solutions of the classical
Yang--Baxter equation and cluster structures on simple Lie groups.
This paper confirms the conjecture for $SL_{5}$. Given a Belavin--Drinfeld
class, we construct the corresponding cluster structure in $\mathcal{O}\left(SL_{5}\right)$,
and show that it satisfies all parts of the conjecture.
\end{abstract}


\maketitle

\section{Introduction}

Since cluster algebras were introduced by Fomin and Zelevinsky (\cite{FZ1}),
the following natural question arose: given an algebraic variety $V$
- can one find a cluster structure in the coordinate ring of $V$? 

Partial answers were given for the Grassmannian $G\left(n,k\right)$
(see \cite{Scott}) and for simple Lie groups (\cite{BFZ}). For a
while, it was assumed that for a given Lie group $\mathcal{G}$ there
is a unique cluster structure in the coordinate ring $\mathbb{C}\left[\mathcal{G}\right]$.
Gekhtman Shapiro and Vainshtein showed in \cite{gekhtman2012cluster}
that there could be multiple cluster structure in $\mathbb{C}\left[\mathcal{G}\right]$.
They called the newly discovered structures ``exotic'', as opposed
to the one already known, which was sometimes referred to as the ``standard''
structure. This is due to the fact that it corresponds to the standard
Sklyanin bracket on $\mathcal{G}$. They also conjectured that these
cluster structures correspond to the Belavin--Drinfeld classification
of solutions to the classical Yang -- Baxter equation (CYBE). This
paper confirms the conjecture for $\mathcal{G}=SL_{5}$.

\subsection{Cluster structures and cluster algebras}

\
Let $\{z_{1},\ldots,z_{m}\}$ be a set of independent variables, and
let $S$ denote the ring of Laurent polynomials generated by $z_{1},\ldots,z_{m}$

\[
S=\mathbb{Z}\left[z_{1}^{\pm1},\ldots,z_{m}^{\pm1}\right].
\]
 (Here and in what follows $z^{\pm1}$ stands for $z,z^{-1}$). The
\emph{ambient field} $\mathcal{F}$ is the field of rational functions
in $n$ independent variables (distinct from $z_{1},\ldots,z_{m}$),
with coefficients in the field of fractions of $S$ (If $m=0$ this
field is just $\mathbb{Q}$).

Let $\tilde{B}=\left(b_{i,j}\right)$ be an integer $n\times\left(n+m\right)$
matrix, whose principal part $B$ is skew symmetric. The
variables $x_{1},\ldots x_{n}$ are called cluster variables, while
$x_{n+i}=z_{i}$, $1\leq i\leq m$ are called stable (or frozen) variables.
The set $\mathbf{x}=\left(x_{1},\ldots,x_{n}\right)$ is called a
cluster, and the set $\tilde{\mathbf{x}}=\left(x_{1},\ldots,x_{n+m}\right)$
is called an extended cluster.

The adjacent cluster in direction $k$ is $\mathbf{x}_{k}=\mathbf{x}\setminus{x_{k}}\cup\left\{ x_{k}'\right\} $
, where $x_{k}'$ is defined by the \emph{exchange relation} 
\begin{equation}
x_{k}x_{k}'=\prod_{b_{k,j}>0}x_{j}^{b_{k,j}}+\prod_{b_{k,j}<0}x_{j}^{-b_{k,j}}.\label{eq:ExRltn}
\end{equation}
 A \emph{mutation} $\mu_{k}\left(B\right)$ of the matrix $B=\left(b_{i,j}\right)$
in direction $k$ is defined 
\[
\mu_{k}\left(B\right)_{i,j}=\begin{cases}
-b_{ij} & \text{ if }i=k\text{ or }j=k\\
b_{ij}+\frac{1}{2}\left(\left|b_{ik}\right|b_{kj}+b_{ik}\left|b_{kj}\right|\right) & \text{ otherwise.}
\end{cases}
\]

A pair $\left(\tilde{\mathbf{x}},\tilde{B}\right)$is called \emph{a
seed}. The adjacent seed in direction $k$ is then $\left(\mathbf{\tilde{x}}_{k},\mu_{k}\left(\tilde{B}\right)\right).$
Two seeds are said to be mutation equivalent if they can be connected
by a sequence of pairwise adjacent seeds. 

\
Given a seed $\Sigma=(\tilde{\mathbf{x}},\tilde{B})$, the \emph{cluster
structure} $\mathcal{C}(\Sigma)$ in $\mathcal{F}$ is the set of
all seeds that are mutation equivalent to $\Sigma$. 

Let $\Sigma$ be a seed as above, and $\mathbb{A}=\mathbb{Z}\left[x_{n+1},\ldots,x_{n+m}\right]$.
The \emph{cluster algebra} $\mathcal{A}=\mathcal{A}(\mathcal{C})=\mathcal{A}(\tilde{B})$
associated with the seed $\Sigma$ is the $\mathbb{A}$-subalgebra
of $\mathcal{F}$ generated by all cluster variables in all seeds
in $\mathcal{C}\left(\tilde{B}\right)$. The \emph{upper cluster algebra
}$\mathcal{\overline{A}}=\mathcal{\overline{A}}(\mathcal{C})=\mathcal{\overline{A}}(\tilde{B})$\emph{
}is the intersection of the rings of Laurent polynomials over $\mathbb{A}$
in cluster variables taken over all seeds in $\mathcal{C}\left(\tilde{B}\right)$.
The famous \emph{Laurent phenomenon} \cite{FZ2} claims the inclusion
$\mathcal{A}(\mathcal{C})\subseteq\mathcal{\overline{A}}(\mathcal{C})$. 

\
It is convenient do describe $\mathcal{C}\left(\tilde{B}\right)$
(or the matrix $\tilde{B}$) by the exchange quiver $Q\left(\tilde{B}\right)$:
it has $n+m$ vertices, each corresponds to a variable $x_{k}$, and
there is a directed edge $i\to j$ with weight $w$ if $\tilde{B}_{ij}=w>0$
.

\
Let $V$ be a quasi-affine variety over $\mathbb{C}$ , $\mathbb{C}\left(V\right)$
be the field of rational functions on $V$ , and $\mathcal{O}\left(V\right)$
be the ring of regular functions on $V$ . Let $\mathcal{C}$ be a
cluster structure in $\mathcal{F}$ as above. Assume that $\left\{ f_{1},\ldots,f_{n+m}\right\} $
is a transcendence basis of $\mathbb{C}\left(V\right)$. Then the
map $\varphi:x_{i}\to f_{i}$ , $1\leq i\leq n+m$, can be extended
to a field isomorphism $\varphi:\mathcal{F}_{\mathbb{C}}\to\mathbb{C}$
, where $\mathcal{F}_{\mathbb{C}}=\mathcal{F}\otimes\mathbb{C}$ is
obtained from $\mathcal{F}$ by extension of scalars. The pair$\left(\mathcal{C},\varphi\right)$
is called a cluster structure in $\mathbb{C}\left(V\right)$ (or just
a cluster structure on $V$), $\left\{ f_{1},\ldots,f_{n+m}\right\} $
is called an extended cluster in $\left(\mathcal{C},\varphi\right)$.
Sometimes we omit direct indication of $\varphi$ and say that $C$
is a cluster structure on $V$ . A cluster structure $\left(\mathcal{C},\varphi\right)$
is called regular if $\varphi\left(x\right)$ is a regular function
for any cluster variable $x$. The two algebras defined above have
their counterparts in $\mathcal{F}_{\mathbb{C}}$ obtained by extension
of scalars; they are denoted $\mathcal{A}_{\mathbb{C}}$ and $\overline{\mathcal{A}}_{\mathbb{C}}$
. If, moreover, the field isomorphism $\varphi$ can be restricted
to an isomorphism of $\mathcal{A}_{\mathbb{C}}$ (or $\overline{\mathcal{A}}_{\mathbb{C}}$) 
and $\mathcal{O}\left(V\right)$, we say that $\mathcal{A}_{\mathbb{C}}$
(or $\overline{\mathcal{A}}_{\mathbb{C}}$) is \emph{naturally isomorphic}
to $\mathcal{O}\left(V\right)$. 

The following statement is a weaker analogue of Proposition 3.37 in
\cite{GSV}. 
\begin{prop}
\label{prop:ACNatIsoO(V)}Let $V$ be a Zariski open subset in $\mathbb{C}^{n+m}$
and $\left(\mathcal{C}=\mathcal{C}\left(\tilde{B}\right),\varphi\right)$
be a cluster structure in $\mathbb{C}\left(V\right)$ with $n$ cluster
and $m$ stable variables such that 
\begin{enumerate}
\item $\rank\tilde{B}=n$;\label{prop:ACNIOVCond1}
\item there exists an extended cluster $\tilde{\mathbf{x}}=\left(x_{1},\ldots,x_{n+m}\right)$
in $\mathcal{C}$ such that $\varphi\left(x_{i}\right)$ is regular
on $V$ for $i\in\left[n+m\right]$;\label{prop:ACNIOVCond2}
\item for any cluster variable $x'_{k},\ k\in\left[n\right]$, obtained
by applying the exchange relation (\ref{eq:ExRltn}) to $\tilde{\mathbf{x}}$,
$\varphi\left(x'_{k}\right)$ is regular on $V$;\label{prop:ACNIOVCond3}
\item for any stable variable $x_{n+i}$, $i\in\left[m\right]$, the function
$\varphi\left(x_{n+i}\right)$ vanishes at some point of $V$;\label{prop:ACNIOVCond4}
\item each regular function on $V$ belongs to $\varphi\left(\overline{\mathcal{A}}_{\mathbb{C}}\left(\mathcal{C}\right)\right)$.\label{prop:ACNIOVCond5} 
\end{enumerate}
Then \textup{$\overline{\mathcal{A}}_{\mathbb{C}}\left(\mathcal{C}\right)$
is naturally isomorphic to }$\mathcal{O}\left(V\right)$.
\end{prop}

\subsection{Compatible Poisson brackets}

Let $\left\{ \cdot,\cdot\right\} $ be a Poisson bracket on $\mathcal{F}$.
Two elements $f_{1},f_{2}\in\mathcal{F}$ are \emph{log - canonical}
with respect to $\left\{ \cdot,\cdot\right\} $
if there exists an integer $\omega$ s.t. $\left\{ f_{1},f_{2}\right\} =\omega f_{1}f_{2}$.
A set $F\subseteq\mathcal{F}$ is log - canonical if every pair in
$F$ is log - canonical. 

We say that $\left\{ \cdot,\cdot\right\} $ is \emph{compatible} with
the cluster structure $\mathcal{C}\left(\tilde{B}\right)$ if every
cluster is log - canonical w.r.t. $\left\{ \cdot,\cdot\right\} $,
that is for every pair $x_{i},x_{j}$ in an extended cluster $\mathbf{x}$ there
exists an integer $\omega_{i,j}$such that 
\begin{equation}
\left\{ x_{i},x_{j}\right\} =\omega_{i,j}x_{i}x_{j}
\end{equation}
 The matrix $\Omega^{\mathbf{x}}=\left(\omega_{i,j}\right)$ is called
the \emph{coefficient matrix} of $\left\{ \cdot,\cdot\right\} $ (in
the basis $\mathbf{x}$); clearly, $\Omega^{\mathbf{x}}$ is skew
symmetric.

A complete characterization of Poisson brackets compatible with a
given cluster structure $\mathcal{C}\left(\tilde{B}\right)$ in the
case $\rank B=n$ is given in \cite{GSV1}, see also \cite[Ch. 4]{GSV}.
In particular, the following statement is an immediate corollary of
Theorem 1.4 in \cite{GSV1}:
\begin{prop}
If $\rank B=n$ then a Poisson bracket is compatible with $\mathcal{C}\left(\tilde{B}\right)$
if and only if its coefficient matrix $\Omega^{\mathbf{x}}$ satisfies
$B\Omega^{\mathbf{x}}$=$\left[D0\right]$, where $D$ is a diagonal
matrix. \label{prop:CompPsnBrkt}
\end{prop}

\subsection{Poisson--Lie groups and R-matrices}

A Lie group $\mathcal{G}$ with a Poisson bracket $\{\cdot,\cdot\}$
is called a \emph{Poisson--Lie group} if the multiplication map $\mu:\mathcal{G}\times\mathcal{G}\to\mathcal{G}$,
$\mu:(x,y)\mapsto xy$ is Poisson. That is, $\mathcal{G}$ with a
Poisson bracket $\{\cdot,\cdot\}$ is a Poisson--Lie group if 
\[
\{f_{1},f_{2}\}(xy)=\{\rho_{y}f_{1},\rho_{y}f_{2}\}(x)+\{\lambda_{x}f_{1},\lambda_{x}f_{2}\}(y),
\]
 where $\rho_{y}$ and $\lambda_{x}$ are, respectively, right and
left translation operators on $\mathcal{G}$.

Following \cite{reyman1994group}, recall the construction of \emph{the
Drinfeld double} of a Lie algebra $\mathfrak{g}$ with the Killing
form $\left\langle \ ,\ \right\rangle $ : define $D\left(\mathfrak{g}\right)=\mathfrak{g}\oplus\mathfrak{g}$
, with an invariant nondegenerate bilinear form 
\[
\left\langle \left\langle \left(\xi,\eta\right),\left(\xi',\eta'\right)\right\rangle \right\rangle =\left\langle \xi,\xi'\right\rangle -\left\langle \eta,\eta'\right\rangle .
\]
 Define subalgebras $\mathfrak{d}_{\pm}$ of $D\left(\mathfrak{g}\right)$
by 
\[
\mathfrak{d}_{+}=\left\{ \left(\xi,\xi\right):\xi\in\mathfrak{g}\right\} ,\quad\mathfrak{d_{-}}=\left\{ \left(R_{+}\left(\xi\right),R_{-}\left(\xi\right)\right):\xi\in\mathfrak{g}\right\} ,
\]
 where $R_{\pm}\in\End\mathfrak{g}$ are defined for any R-matrix
$r$ by 
\[
\left\langle R_{+}\left(\eta\right),\zeta\right\rangle =-\left\langle R_{-}\left(\zeta\right),\eta\right\rangle =\left\langle r,\eta\otimes\zeta\right\rangle .
\]

Let $\mathfrak{g}$ be the Lie algebra of a Lia group $\mathcal{G}$
with a nondegenerate invariant bilinear form $(\ ,\ )$, and let $\mathfrak{t}\in\mathfrak{g}\otimes\mathfrak{g}$
be the corresponding Casimir element. For an element $r=\sum_{i}a_{i}\otimes b_{i}\in\mathfrak{g}\otimes\mathfrak{g}$
denote 
\[
\left[\left[r,r\right]\right]=\sum_{i,j}\left[a_{i},a_{j}\right]\otimes b_{i}\otimes b_{j}+\sum_{i,j}a_{i}\otimes\left[b_{i},a_{j}\right]\otimes b_{j}+\sum_{i,j}a_{i}\otimes a_{j}\otimes\left[b_{i},b_{j}\right]
\]
 and $r^{21}=\sum_{i}b_{i}\otimes a_{i}$.

A \emph{classical R-matrix} is an element $r\in\mathfrak{g}\otimes\mathfrak{g}$
that satisfies the\emph{ classical Yang- Baxter equation} (CYBE)
\begin{equation}
\left[\left[r,r\right]\right]=0
\end{equation}
 together with the condition 
\begin{equation}
r+r^{21}=\mathfrak{t}
\end{equation}

A \emph{classical R-matrix} induces Poisson bracket on $\mathcal{G}$:
choose a basis $\{I_{\alpha}\}$ in $\mathfrak{g}$, and let $\partial_{\alpha}$
(resp., $\partial_{\alpha}^{\prime}$) be the right (resp., left)
invariant vector field on $\mathcal{G}$ whose value at the unit element
is $I_{\alpha}$. Let $r=\sum_{\alpha,\beta}r_{\alpha,\beta}I_{\alpha}\otimes I_{\beta}$.
Then 
\begin{equation}
\{f_{1},f_{2}\}=\sum_{\alpha,\beta}r_{\alpha,\beta}\left(\partial_{\alpha}f_{1}\partial_{\beta}f_{2}-\partial_{\alpha}^{\prime}f_{1}\partial_{\beta}^{\prime}f_{2}\right)\label{eq:sklnPB}
\end{equation}
 defines a Poisson bracket on $\mathcal{G}$. This bracket is called
the \emph{Sklyanin bracket} corresponding to $r$.

The classification of classical R-matrices for simple complex Lie
groups was given by Belavin and Drinfeld in \cite{BDSolCYBE}. 

Let $\mathfrak{g}$ be a simple complex Lie algebra with a fixed nondegenerate
invariant symmetric bilinear form $(\ ,\ )$. Fix a Cartan subalgebra
$\mathfrak{h}$, a root system $\Phi$ of $\mathfrak{g}$, and a set
of positive roots $\Phi^{+}$. Let $\Delta\subseteq\Phi^{+}$ be a
set of positive simple roots. 

A Belavin--Drinfeld triple is two subsets $\Gamma_{1},\Gamma_{2}$
of $\Delta$ and an isometry $\gamma:\Gamma_{1}\to\Gamma_{2}$ nilpotent
in the following sense: for every $\alpha\in\Gamma_{1}$ there exists
$m\in\mathbb{N}$ such that $\gamma^{j}(\alpha)\in\Gamma_{1}$ for
$j=0,\ldots,m\lyxmathsym{\textminus}1$, but $\gamma^{m}(\alpha)\notin\Gamma_{1}$
. The isometry $\gamma$ extends in a natural way to a map between
root systems generated by $\Gamma_{1},\Gamma_{2}$. This allows one
to define a partial ordering on $\Phi$: $\alpha\prec\beta$ if $\beta=\gamma^{j}\left(\alpha\right)$
for some $j\in\mathbb{N}$.

Select root vectors $e_{\alpha}\in\mathfrak{g}$ satisfying $\left(e_{\alpha},e_{-\alpha}\right)=1$.
According to the Belavin--Drinfeld classification, the following is
true (see, e.g., \cite[Chap. 3]{chriprsly}). 
\begin{prop}
(i) Every classical R-matrix is equivalent (up to an action of $\sigma\otimes\sigma$
where $\sigma$ is an automorphism of $\mathfrak{g}$) to 
\begin{equation}
r=r_{0}+\sum_{\alpha\in\Phi^{+}}e_{-\alpha}\otimes e_{\alpha}+\sum_{\begin{subarray}{c}
\alpha\prec\beta\\
\alpha,\beta\in\Phi^{+}
\end{subarray}}e_{-\alpha}\otimes e_{\beta}\label{eq:RmtxCons}
\end{equation}

(ii) $r_{0}\in\mathfrak{h}\otimes\mathfrak{h}$ in (\ref{eq:RmtxCons})
satisfies 
\begin{equation}
\left(\gamma\left(\alpha\right)\otimes Id\right)r_{0}+\left(Id\otimes\alpha\right)r_{0}=0\label{eq:r0Cond1}
\end{equation}
 for any $\alpha\in\Gamma_{1}$, and 
\begin{equation}
r_{0}+r_{0}^{21}=\mathfrak{t}_{0},\label{eq:r0Cond2}
\end{equation}
 where $\mathfrak{t}_{0}$ is the $\mathfrak{h}\otimes\mathfrak{h}$
component of $\mathfrak{t}$ .

(iii) Solutions $r_{0}$ to (\ref{eq:r0Cond1}),(\ref{eq:r0Cond2})
form a linear space of dimension $k_{T}=\left|\Delta\setminus\Gamma_{1}\right|$. 
\end{prop}
We say that two classical R-matrices that have a form (\ref{eq:RmtxCons})
belong to the same \emph{Belavin--Drinfeld class} if they are associated
with the same Belavin--Drinfeld triple. 

Given a Belavin--Drinfeld triple $T$ for $\mathcal{G}$, define the
torus $\mathcal{H}_{T}=\exp\mathfrak{h}_{T}\subset\mathcal{G}$, 
where 
\[
 \mathfrak{h}_T =\left\{
 h\in \mathfrak{h}:\alpha (h)=\beta(h)
 \text{ if }
 \alpha\prec \beta\right\}
\]

In \cite{gekhtman2012cluster} Gekhtman, Shapiro and Vainshtein give
the following conjecture:

\begin{conjecture}
\label{Conj:GSV-BD-CS}Let $\mathcal{G}$ be a simple complex Lie
group. For any Belavin--Drinfeld triple $T=(\Gamma_{1},\Gamma_{2},\gamma)$
there exists a cluster structure $\mathcal{C}_{T}$ on $\mathcal{G}$
such that \end{conjecture}
\begin{enumerate}
\item the number of stable variables is $2k_{T}$, and the corresponding
extended exchange matrix has a full rank. \label{Conj:NumStbVar}
\item $\mathcal{C}_{T}$ is regular, and the corresponding upper cluster
algebra $\overline{\mathcal{A}}_{\mathbb{C}}(\mathcal{C}_{T})$ is
naturally isomorphic to $\mathcal{O}(\mathcal{G})$; \label{enu:A(C)eqlsO(G)}
\item the global toric action of $\left(\mathbb{C}^{*}\right)^{2k_{T}}$
on $\mathbb{C}\left(\mathcal{G}\right)$ is generated by the action
of $\mathcal{H}_{T}\otimes\mathcal{H}_{T}$ on $\mathcal{G}$ given
by $\left(H_{1},H_{2}\right)\left(X\right)=H_{1}XH_{2}$ ; \label{conj:global-toric}
\item for any solution of CYBE that belongs to the Belavin--Drinfeld class
specified by $T$ , the corresponding Sklyanin bracket is compatible
with $\mathcal{C}_{T}$; 
\item a Poisson--Lie bracket on $\mathcal{G}$ is compatible with $\mathcal{C}_{T}$
only if it is a scalar multiple of the Sklyanin bracket associated
with a solution of CYBE that belongs to the Belavin--Drinfeld class
specified by $T$. 
\end{enumerate}
The conjecture was proved for the Belavin--Drinfeld class $\Gamma_{1}=\Gamma_{2}=\emptyset$.
This trivial triple corresponds to the standard Poisson--Lie bracket.
We call the cluster structures associated with the non-trivial Belavin--Drinfeld
data \emph{exotic}. The conjecture is true also for all exotic cluster structures
on $SL_{n}$ when $n\leq4$ (see \cite{gekhtman2012cluster}). The
Cremmer--Gervais case is ``the furthest'' from the standard case
(because here $\left|\Gamma_{1}\right|$ is maximal). The conjecture
was proved for this case in \cite{gekhtman2013exotic}. This paper
covers $SL_{5}$.

\section{Exotic cluster structures\label{sec:Exotic-cluster-structures}}

The main result of this paper is the following theorem:
\begin{thm}
Conjecture \ref{Conj:GSV-BD-CS} is true for $\mathcal{G}=SL_{5}$
and any Belavin--Drinfeld triple $T=\left(\Gamma_{1},\Gamma_{2},\gamma\right)$.
\end{thm}

This theorem will be proved by constructing a structure $\mathcal{C}_{T}$
on $\mathcal{G}$ that satisfy statements 1-5 of the conjecture.

Consider the two following isomorphisms of the Belavin--Drinfeld data
on $SL_{5}$: the first one transposes $\Gamma_{1}$ and $\Gamma_{2}$
and reverses the direction of $\gamma$, while the second one takes
each root $\alpha_{j}$ to $\alpha_{\omega_{0}\left(j\right)}$ .
These two isomorphisms correspond to the automorphisms of $SL_{5}$
given by $X\mapsto-X^{t}$ and $X\mapsto\omega_{0}X\omega_{0}$, respectively.
Since we consider $R$-matrices up to an action of $\sigma\otimes\sigma$,
in what follows we do not distinguish between Belavin--Drinfeld triples
obtained one from the other via these isomorphisms. 

$SL_{5}$ has four simple roots, and the Dynkin diagram is given in
Figure \ref{fig:Dnkndig4}.

\begin{figure}
\begin{centering}
\includegraphics[scale=0.3]{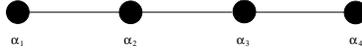}
\par\end{centering}

\caption{The Dynkin diagram of $SL_{5}$}
 \label{fig:Dnkndig4}
\end{figure}

Therefore, the Belavin--Drinfeld triples are (up to the above isomorphisms):
\begin{enumerate}
\item $\Gamma_{1}=\left\{ \emptyset\right\} ,\Gamma_{2}=\left\{ \emptyset\right\} $\label{BDcs:trv}
\
\item $\Gamma_{1}=\left\{ \alpha_{1}\right\} ,\Gamma_{2}=\left\{ \alpha_{2}\right\} ,\gamma:\alpha_{1}\to\alpha_{2}$ 
\item $\Gamma_{1}=\left\{ \alpha_{1}\right\} ,\Gamma_{2}=\left\{ \alpha_{3}\right\} ,\gamma:\alpha_{1}\to\alpha_{3}$ 
\item $\Gamma_{1}=\left\{ \alpha_{1}\right\} ,\Gamma_{2}=\left\{ \alpha_{4}\right\} ,\gamma:\alpha_{1}\to\alpha_{4}$ 
\item $\Gamma_{1}=\left\{ \alpha_{2}\right\} ,\Gamma_{2}=\left\{ \alpha_{3}\right\} ,\gamma:\alpha_{2}\to\alpha_{3}$ 
\item $\Gamma_{1}=\left\{ \alpha_{1},\alpha_{2}\right\} ,\Gamma_{2}=\left\{ \alpha_{2},\alpha_{3}\right\} ,\gamma:\alpha_{i}\to\alpha_{i+1}$ 
\item $\Gamma_{1}=\left\{ \alpha_{1},\alpha_{2}\right\} ,\Gamma_{2}=\left\{ \alpha_{3},\alpha_{4}\right\} ,\gamma:\alpha_{i}\to\alpha_{i+2}$ 
\item $\Gamma_{1}=\left\{ \alpha_{1},\alpha_{3}\right\} ,\Gamma_{2}=\left\{ \alpha_{2},\alpha_{4}\right\} ,\gamma:\alpha_{i}\to\alpha_{i+1}$ 
\item $\Gamma_{1}=\left\{ \alpha_{1},\alpha_{3}\right\} \Gamma_{2}=\left\{ \alpha_{2},\alpha_{4}\right\} ,\gamma:\alpha_{i}\to\alpha_{5-i}$ 
\item $\Gamma_{1}=\left\{ \alpha_{1},\alpha_{3}\right\} ,\Gamma_{2}=\left\{ \alpha_{1},\alpha_{4}\right\} ,\gamma:\alpha_{i}\to\alpha_{i+3\pmod5}$ 
\item $\Gamma_{1}=\left\{ \alpha_{1},\alpha_{2}\right\} ,\Gamma_{2}=\left\{ \alpha_{3},\alpha_{4}\right\} ,\gamma:\alpha_{i}\to\alpha_{5-i}$
\label{BDcs:SpclCs}
\item $\Gamma_{1}=\left\{ \alpha_{1},\alpha_{2},\alpha_{3}\right\} ,\mapsto\left\{ \alpha_{2},\alpha_{3},\alpha_{4}\right\} ,\gamma:\alpha_{i}\to\alpha_{i+1}$
\label{BDcs:CG} 
\item $\Gamma_{1}=\left\{ \alpha_{1},\alpha_{2},\alpha_{4}\right\} ,\Gamma_{2}=\left\{ \alpha_{3},\alpha_{4},\alpha_{1}\right\} ,\gamma:\alpha_{i}\to\alpha_{i+2\pmod5}$ 
\end{enumerate}

Slightly abusing the notation, we sometime refer to a root $\alpha_{i}\in\Delta$
just as $i,$ and write $\gamma:i\mapsto j$ instead of $\gamma:\alpha_{i}\mapsto\alpha_{j}$.
We say that a Belavin--Drinfeld triple $T=\left(\Gamma_{1},\Gamma_{2},\gamma\right)$
is \emph{orientable} unless there is a pair of adjacent roots
$i,i+1$ in $\Gamma_{1}$ such that $\gamma\left(i+1\right)+1=\gamma\left(i\right)$. 

Case \ref{BDcs:trv} (the trivial case) corresponds to the standard
cluster structure on $SL_{5}$ described in \cite{BFZ} (see also
\cite{gekhtman2012cluster}). 

Case \ref{BDcs:CG} is the Cremmer--Gervais case, and it is covered
in \cite{gekhtman2013exotic}. Case \ref{BDcs:SpclCs} is the only
non-orientable triple. It will be treated separately in Section \ref{sec:NonOrientable}.
The following discussion concerns all orientable Belavin--Drinfeld
triples.

\

\subsection{Constructing a log canonical set for the orientable cases}

Let $T=(\Gamma_{1},\Gamma_{2},\gamma)$
be an orientable Belavin--Drinfeld triple for $SL_{5}$. Denote by
$\left\{ \cdot,\cdot\right\} _{T}$ the corresponding Poisson--Lie
bracket. We are going to construct a set of matrices $\mathcal{M}$,
so that the set of all leading principal minors of these matrices
will be log canonical with respect to $\left\{ \cdot,\cdot\right\} _{T}$.
This set will be the basis of the corresponding exotic cluster structure. 

\
We start our construction with an element
$\left(X,Y\right)$ in the double $D(SL_{5})$
. The building blocks of our matrices are submatrices of the matrices
$X$ and $Y$ of the following form: set $k\in{0,\ldots4}$. A building
block is obtained either by deleting the first $k$ rows and last
$k$ columns of $X$, or by deleting the first $k$ columns and the
last $k$ rows of $Y$. Using the notation $X_{R}^{C}$ for the submatrix
of $X$ with set of rows indices $R$ and set of column indices $C$,
we have two kinds of these blocks: $U{}_{i}=X_{\left[i\ldots5\right]}^{\left[1\ldots6-i\right]}$
and $V_{j}=Y_{\left[1\ldots6-j\right]}^{\left[j\ldots5\right]}$ . 

The lower right corner of these matrices has either $x_{5\xi}$ or
$y_{\xi5}$, where $\xi=6-i$ or $\xi=6-j$, respectively. For a matrix
$A$ we will denote this number by $\xi\left(A\right)$.

Define 
\[
\sigma\left(A\right)=\begin{cases}
1, & \text{if }A\text{ is }U_{i},\\
-1, & \text{if }A\text{ is }V_{j}, 
\end{cases}
\]

and $\overline{\sigma}\left(A\right)=\frac{1}{2}\left(3-\sigma\left(A\right)\right)$
, so 
\[
\overline{\sigma}\left(A\right)=\begin{cases}
1 & \text{if }\sigma\left(A\right)=1\\
2 & \text{if }\sigma\left(A\right)=-1.
\end{cases}
\]

Let $A=W_{Q}^{P}$, where $W$ is a matrix $X$ or $Y$, and $P,Q$
are subsets of $\left\{ 1,\ldots,5\right\} $ of the form $\left\{ 1,\ldots,k\right\} $
or $\left\{ \ell,\ldots,5\right\} .$ An \emph{extension }of a submatrix
$A$ by a number $t$ is adding rows or columns of the matrix $W$
to $A$ according to the following rule:

if $t>0$, the set $\left\{ 1,\ldots k\right\} $ becomes $\left\{ 1,\ldots,k+t\right\} $.
If $t<0$ the set $\left\{ \ell,\ldots,5\right\} $ becomes $\left\{ \ell-t,\ldots,5\right\} $.
Denote the extended matrix by $A\left(t\right)$. Thus for a positive
integer $t$, a matrix of type $U_{i}$ can be extended by $t$ to
$U_{i}\left(t\right)=X_{\left[i\ldots5\right]}^{\left[1\ldots6-i+t\right]}$
or by $-t$ to $U_{i}\left(-t\right)=X_{\left[i-t\ldots5\right]}^{\left[1\ldots6-i\right]}$.
Similarly, extending $V_{j}$ by $t$ gives $V_{j}\left(t\right)=Y_{\left[1\ldots6-j+t\right]}^{\left[j\ldots5\right]}$
and extending it by $-t$ gives $V_{j}\left(-t\right)=Y_{\left[1\ldots6-j\right]}^{\left[j-t\ldots5\right]}$.
Clearly, for a given submatrix $A$ of type $V$ or $U$ the two possible
(positive and negative) directions of extension of $A$ are determined
by $\sigma\left(A\right)$. 

An \emph{extended block diagonal }matrix is a block diagonal matrix
whose blocks are submatrices of $X$ and $Y$ that have been extended
as above, with two restrictions: the first block was not extended
in the negative direction and the last block was not extended in the
positive direction. These two conditions assure the matrix is still
a square matrix. An example is given in Figure \ref{Fig:ExtBlckDiagMat}:
the striped rectangles are extensions of the blocks $B_{k}$. A block
of an extended block diagonal matrix is an extended submatrix of $X$
or $Y$. Such a block is not necessarily square, but may still be
extended.

\begin{figure}
\
\begin{centering}
\includegraphics[scale=0.6]{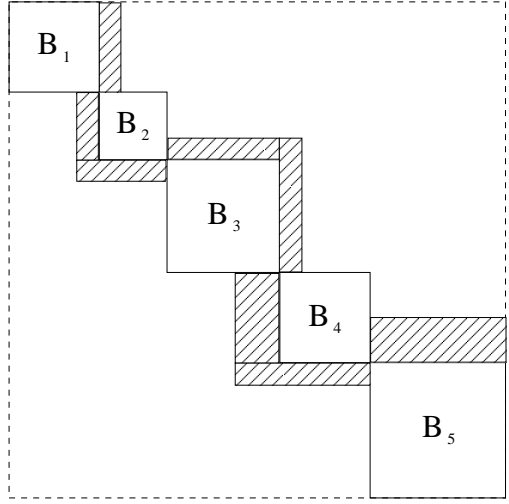}
\par\end{centering}

\
\caption{An extended block diagonal matrix }

\
\label{Fig:ExtBlckDiagMat}\
\end{figure}

Let $A$ be an extended block diagonal matrix with two blocks or more,
$B_{1},\ldots,B_{k}$. We define $\xi\left(A\right)=\xi\left(B_{k}\right),\sigma\left(A\right)=\sigma\left(B_{k}\right)$
and $\overline{\sigma}\left(A\right)=\overline{\sigma}\left(B_{k}\right)$.
Extending such a matrix $A$ is simply extending the last (lower)
block $B_{k}$.

\emph{Gluing }a matrix $A_{2}$ to a matrix $A_{1}$is obtaining a
new block diagonal matrix $A$ with blocks $A_{1}$ and $A_{2}$.
That is, $A=\left[\begin{array}{cc}
A_{1} & 0\\
0 & A_{2}
\end{array}\right]$. 

\
Last, we need to define the extension number: for a submatrix $A$,
we look at the root $\alpha=\xi\left(A\right)$. We look at the subgraph
obtained by intersecting the Dynkin diagram of $SL_{5}$ and $\Gamma_{\overline{\sigma}\left(A\right)}$.
The connected component of $\alpha$ in this intersection is a path
$\left(p_{1}<\cdots<\alpha\cdots<p_{k}\right)$. Let $t^{+}\left(A\right)$
be the number of vertices in the path $\left(\alpha,\ldots,p_{k}\right)$,
i.e., $t^{+}\left(A\right)=p_{k}-\alpha+1,$ and let $t^{-}\left(A\right)$
be the number of vertices in the path $\left(p_{1},\ldots,\alpha\right),$
that is $t^{-}\left(A\right)=\alpha-p_{1}+1$. Note that $t^{+}\left(A\right)$
and $t^{-}\left(A\right)$ are set to $1$ if $\alpha$ is the maximal
or minimal root in this subgraph, respectively.

\
To define the set $\mathcal{M}$, start with the set of all $V_{j}$
with $j-1\notin\Gamma_{2}$ and all $U_{i}$ with $i-1\notin\Gamma_{1}$
. 

To every matrix $A$ in this set, apply the following steps:
\
\begin{enumerate}
\item Put $A_{0}=A$.
\
\item If $\xi\left(A_{i}\right)\not\in\Gamma_{\overline{\sigma}\left(A_{i}\right)}$,
then stop the process and add $A_{i}$ to $\mathcal{M}$.\label{Alg:stp:chkstp}
\item Otherwise, if $\sigma\left(A_{i}\right)=1$, define $A_{i}^{+}=V_{\gamma\left(\xi\left(A_{i}\right)\right)+1}$.
\\
If $\sigma\left(A_{i}\right)=-1$, define $A_{i}^{+}=U_{\gamma^{-1}\left(\xi\left(A_{i}\right)\right)+1}$.
\item Glue the matrix $A_{i}^{+}$ to $A_{i}$.
\item Extend $A_{i}$ by $t^{+}\left(A_{i}\right)$, and extend $A_{i}^{+}$
by $t^{-}\left(A_{i}\right)$. Note that this matrix is now extended
block diagonal.\label{Alg:stp:extending}
\item Let $A_{i+1}$ be the matrix obtained in step \ref{Alg:stp:extending}
. Go back to step \ref{Alg:stp:chkstp} with $i=i+1$.
\end{enumerate}
Eventually, we get a set of matrices which are all either just submatrices
of $X$ and $Y$ or extended block diagonal matrices. We will denote this
set by $\mathcal{M}$, and  use it 
to define our log canonical functions.
\begin{example}
In the case $T=\left(\left\{ 1,2,4\right\} ,\left\{ 1,3,4\right\} ,\gamma:i\mapsto i+2\pmod5\right)$
the construction of the set $\mathcal{M}$ starts with submatrices
\begin{eqnarray*}
 &  & M_{1}=X,\ M=Y,\\
 &  & M_{3}=\left[\begin{array}{ccc}
y_{{13}} & y_{{14}} & y_{{15}}\\
y_{{23}} & y_{{24}} & y_{{25}}\\
y_{{33}} & y_{{34}} & y_{{35}}
\end{array}\right],\\
 &  & M_{4}=\left[\begin{array}{cc}
x_{{41}} & x_{{42}}\\
x_{{51}} & x_{{52}}
\end{array}\right].
\end{eqnarray*}
 For $A_{0}=X$ and $A_{0}=Y$ the process stops at step \ref{Alg:stp:chkstp},
because in both cases $\xi\left(A_{0}\right)=5$, and $\xi\left(A_{0}\right)\not\in\Gamma_{\overline{\sigma}\left(A_{0}\right)}$
(for either $\overline{\sigma}\left(A_{0}\right)=1$ or $\overline{\sigma}\left(A_{0}\right)=2$),
therefore $X,Y\in\mathcal{M}$. Note that this is true for every orientable
Belavin--Drinfeld data $T$. 

When we take $A_{0}=M_{3}$ we have $\overline{\sigma}\left(A_{0}\right)=2$
and $\xi\left(A_{0}\right)=3$, therefore
$\xi\left(A_{0}\right)\in\Gamma_{\overline{\sigma}\left(A_{0}\right)}$
, and we proceed to step 2. Here, since $\sigma\left(A_{0}\right)=-1$
we glue the submatrix $A_{0}^{+}=U_{\gamma^{-1}\left(\xi\left(A_{0}\right)\right)}=U_{2}$
and get the matrix 
\[
\left[\begin{array}{ccccccc}
y_{{13}} & y_{{14}} & y_{{15}} & 0 & 0 & 0 & 0\\
y_{{23}} & y_{{24}} & y_{{25}} & 0 & 0 & 0 & 0\\
y_{{33}} & y_{{34}} & y_{{35}} & 0 & 0 & 0 & 0\\
0 & 0 & 0 & x_{{21}} & x_{{22}} & x_{{23}} & x_{{24}}\\
0 & 0 & 0 & x_{{31}} & x_{{32}} & x_{{33}} & x_{{34}}\\
0 & 0 & 0 & x_{{41}} & x_{{42}} & x_{{43}} & x_{{44}}\\
0 & 0 & 0 & x_{{51}} & x_{{52}} & x_{{53}} & x_{{54}}
\end{array}\right].
\]
 To determine the extension numbers we look at the intersection of
the set $\Gamma_{\overline{\sigma}\left(A_{0}\right)}$ and  the Dynkin
diagram (see Figure \ref{Fig:Dnk134}): the connected component of
the root $3$ in this intersection has two vertices - $3$ and $4$
connected by an edge. Therefore $t^{+}\left(A_{0}\right)=2$ and $t^{-}\left(A_{0}^{+}\right)=1$
. Hence, after extending in step \ref{Alg:stp:extending}, we get
the matrix 

\
\[
A_{1}=\left[\begin{array}{ccccccc}
y_{{13}} & y_{{14}} & y_{{15}} & 0 & 0 & 0 & 0\\
y_{{23}} & y_{{24}} & y_{{25}} & 0 & 0 & 0 & 0\\
y_{{33}} & y_{{34}} & y_{{35}} & x_{11} & x_{12} & x_{13} & x_{14}\\
y_{43} & y_{44} & y_{45} & x_{{21}} & x_{{22}} & x_{{23}} & x_{{24}}\\
y_{53} & y_{54} & y_{55} & x_{{31}} & x_{{3,2}} & x_{{33}} & x_{{34}}\\
0 & 0 & 0 & x_{{41}} & x_{{42}} & x_{{43}} & x_{{44}}\\
0 & 0 & 0 & x_{{51}} & x_{{52}} & x_{{53}} & x_{{54}}
\end{array}\right].
\]
 
\begin{figure}
\
\begin{centering}
\includegraphics[scale=0.3]{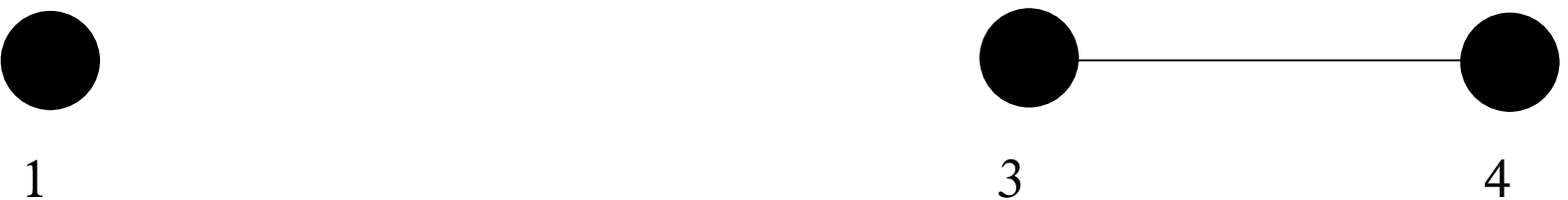}
\par\end{centering}

\
\caption{The intersection of $\Gamma_{2}=\{1,3,4\}$ and the Dynkin diagram }

\label{Fig:Dnk134}
\end{figure}

We return to step 1. with the new matrix $A_{1}$. Now $\xi\left(A_{1}\right)=4$
and $\bar{\sigma}\left(A_{1}\right)=1$, so $\xi\left(A_{1}\right)\in\Gamma_{\overline{\sigma}\left(A_{1}\right)}$
again, and we proceed to step 2. Since $\sigma\left(A_{1}\right)=1$,
we glue the submatrix $A_{1}^{+}=V_{\gamma\left(\xi\left(A_{1}\right)\right)+1}=V_{2}$
to get the matrix 
\[
\left[\begin{array}{ccccc}
A\\
 & x_{{12}} & x_{{13}} & x_{{14}} & x_{{15}}\\
 & x_{{22}} & x_{{23}} & x_{{24}} & x_{{25}}\\
 & x_{{32}} & x_{{3,3}} & x_{{3,4}} & x_{{3,5}}\\
 & x_{{42}} & x_{{43}} & x_{{44}} & x_{{45}}
\end{array}\right].
\]
 Now the connected component of $4$ in the intersection of the set
and $\Gamma_{\overline{\sigma}\left(A_{1}\right)}$ with the Dynkin
diagram has only one vertex $4$. So $t^{+}\left(A_{1}\right)=t^{-}\left(A_{1}^{+}\right)=1,$
and after extending in step \ref{Alg:stp:extending} we have 
\[
A_{2}=\left[\begin{array}{ccccccccccc}
y_{{13}} & y_{{14}} & y_{{15}} & 0 & 0 & 0 & 0 & 0 & 0 & 0 & 0\\
y_{{23}} & y_{{24}} & y_{{25}} & 0 & 0 & 0 & 0 & 0 & 0 & 0 & 0\\
y_{{33}} & y_{{34}} & y_{{35}} & x_{{11}} & x_{{12}} & x_{{13}} & x_{{14}} & x_{{15}} & 0 & 0 & 0\\
y_{{43}} & y_{{44}} & y_{{45}} & x_{{21}} & x_{{22}} & x_{{23}} & x_{{24}} & x_{{25}} & 0 & 0 & 0\\
y_{{53}} & y_{{54}} & y_{{55}} & x_{{31}} & x_{{32}} & x_{{33}} & x_{{34}} & x_{{35}} & 0 & 0 & 0\\
0 & 0 & 0 & x_{{41}} & x_{{42}} & x_{{43}} & x_{{44}} & x_{{45}} & 0 & 0 & 0\\
0 & 0 & 0 & x_{{51}} & x_{{52}} & x_{{53}} & x_{{54}} & x_{{55}} & 0 & 0 & 0\\
0 & 0 & 0 & 0 & 0 & 0 & y_{{11}} & y_{{12}} & y_{{13}} & y_{{14}} & y_{{15}}\\
0 & 0 & 0 & 0 & 0 & 0 & y_{{21}} & y_{{22}} & y_{{23}} & y_{{24}} & y_{{25}}\\
0 & 0 & 0 & 0 & 0 & 0 & y_{{31}} & y_{{32}} & y_{{33}} & y_{{34}} & y_{{35}}\\
0 & 0 & 0 & 0 & 0 & 0 & y_{{41}} & y_{{42}} & y_{{43}} & y_{{44}} & y_{{45}}
\end{array}\right].
\]
 We now go back to step \ref{Alg:stp:chkstp} Again, we look at $\xi\left(A_{2}\right)=4$
and $\bar{\sigma}\left(A_{2}\right)=2$. That means $\xi\left(A_{2}\right)\in\Gamma_{\overline{\sigma}\left(A\right)}$
and we move on. Here $\sigma\left(A_{2}\right)=-1$, and we glue the
submatrix $A_{2}^{+}=U_{\gamma^{-1}\left(\xi\left(A_{2}\right)\right)+1}=U_{3}$. 

The connected component of $4$ in the intersection of the set $\Gamma_{2}$
with the Dynkin diagram has two vertices $3$ and $4$. This gives
the extension numbers $t^{+}\left(A_{2}\right)=1$ and $t^{-}\left(A_{2}\right)=2$,
and the new matrix is 
\[
A_{3}=\left[\begin{array}{cccccccccccccc}
y_{{13}} & y_{{14}} & y_{{15}} & 0 & 0 & 0 & 0 & 0 & 0 & 0 & 0 & 0 & 0 & 0\\
y_{{23}} & y_{{24}} & y_{{25}} & 0 & 0 & 0 & 0 & 0 & 0 & 0 & 0 & 0 & 0 & 0\\
y_{{33}} & y_{{34}} & y_{{35}} & x_{{11}} & x_{{12}} & x_{{13}} & x_{{14}} & x_{{15}} & 0 & 0 & 0 & 0 & 0 & 0\\
y_{{43}} & y_{{44}} & y_{{45}} & x_{{21}} & x_{{22}} & x_{{23}} & x_{{24}} & x_{{25}} & 0 & 0 & 0 & 0 & 0 & 0\\
y_{{53}} & y_{{54}} & y_{{55}} & x_{{31}} & x_{{32}} & x_{{33}} & x_{{34}} & x_{{35}} & 0 & 0 & 0 & 0 & 0 & 0\\
0 & 0 & 0 & x_{{41}} & x_{{42}} & x_{{43}} & x_{{44}} & x_{{45}} & 0 & 0 & 0 & 0 & 0 & 0\\
0 & 0 & 0 & x_{{51}} & x_{{52}} & x_{{53}} & x_{{54}} & x_{{55}} & 0 & 0 & 0 & 0 & 0 & 0\\
0 & 0 & 0 & 0 & 0 & 0 & y_{{11}} & y_{{12}} & y_{{13}} & y_{{14}} & y_{{15}} & 0 & 0 & 0\\
0 & 0 & 0 & 0 & 0 & 0 & y_{{21}} & y_{{22}} & y_{{23}} & y_{{24}} & y_{{25}} & 0 & 0 & 0\\
0 & 0 & 0 & 0 & 0 & 0 & y_{{31}} & y_{{32}} & y_{{33}} & y_{{34}} & y_{{35}} & x_{{11}} & x_{{12}} & x_{{13}}\\
0 & 0 & 0 & 0 & 0 & 0 & y_{{41}} & y_{{42}} & y_{{43}} & y_{{44}} & y_{{45}} & x_{{21}} & x_{{22}} & x_{{23}}\\
0 & 0 & 0 & 0 & 0 & 0 & y_{{51}} & y_{{52}} & y_{{53}} & y_{{54}} & y_{{55}} & x_{{31}} & x_{{32}} & x_{{33}}\\
0 & 0 & 0 & 0 & 0 & 0 & 0 & 0 & 0 & 0 & 0 & x_{{41}} & x_{{42}} & x_{{43}}\\
0 & 0 & 0 & 0 & 0 & 0 & 0 & 0 & 0 & 0 & 0 & x_{{51}} & x_{{52}} & x_{{53}}
\end{array}\right].
\]
Finally, returning to step \ref{Alg:stp:chkstp} with this matrix
we have $\xi\left(A_{3}\right)=3$ and $\bar{\sigma}\left(A_{3}\right)=1$.
Now $\xi\left(A_{3}\right)\not\in\Gamma_{\overline{\sigma}\left(A_{3}\right)}$,
which stops the process. We add the matrix $A_{3}$ to the set $\mathcal{M}$
and turn to the last matrix $m_{4}$.

Starting the algorithm with $M_{4}$ leads to the matrix 
\[
\left[\begin{array}{cccccc}
x_{{41}} & x_{{42}} & x_{{43}} & 0 & 0 & 0\\
x_{{51}} & x_{{52}} & x_{{53}} & 0 & 0 & 0\\
y_{{13}} & y_{{14}} & y_{{15}} & x_{{41}} & x_{{42}} & x_{{43}}\\
y_{{23}} & y_{{24}} & y_{{25}} & x_{{51}} & x_{{52}} & x_{{53}}\\
0 & 0 & 0 & y_{{13}} & y_{{14}} & y_{{15}}\\
0 & 0 & 0 & y_{{23}} & y_{{24}} & y_{{25}}
\end{array}\right],
\]
 which is the last matrix in the set $\mathcal{M}$ .
\end{example}
\
Note that $U_{1}=X$ and $V_{1}=Y$ are always in this set $\mathcal{M}$.
The number of elements in $\mathcal{M}$
is just the number of submatrices we start with. For each root $i-1\not\in\Gamma_{1}$
we have such a submatrix, as well as for each $j-1\notin\Gamma_{2}$.
Adding one for $U_{1}=X$ and $V_{1}=Y$, we end up with $\left|\mathcal{M}\right|=2\left|\Delta\setminus\Gamma_{1}\right|+2=2k_{T}+2$
matrices.

We can now define our log canonical functions: for every $M\in\mathcal{M}$
take all the leading principal minors of $M$. If
we denote the number of rows (and columns) of a matrix $M$ by $s\left(M\right)$,
these are all the functions $\det M_{\left[1\ldots r\right]}^{\left[1\ldots r\right]}$
with $1\leq r\leq s\left(M\right)$. Thus
we have a set of functions $F\left(X,Y\right)$ on $D\left(SL_{5}\right)$.
The projection of $F\left(X,Y\right)$ on the diagonal subgroup can
be viewed as a function $f\left(X\right)=F\left(X,X\right)$ on $SL_{5}$.
Note that after this projection all the minors of $U_{1}=X$ coincide
with those of $V_{1}=Y$.

\
For any of these functions $f\left(X\right)$ consider the lower right
matrix entry of the submatrix associated with $f$ . This entry is
$x_{i,j}$ for some $i,j\in\left[5\right]$. This defines a map $\rho$
from the set of functions to $\left[5\right]\times\left[5\right]$.
\
\begin{prop}
The map $\rho$ is bijective.\label{prop:RhoBjctv}\end{prop}

\begin{proof}
Consider the set $\left\{ U_{i},V_{j}\right\} $ of building blocks
of the matrices of $\mathcal{M}$. Each block was used once: either
as a starting block if $i-1\notin\Gamma_{1}$
or $j-1\notin\Gamma_{2}$, or as a glued block if $i-1\in\Gamma_{1}$
or $j-1\in\Gamma_{2}$. Since the main diagonals of these blocks are
exactly all the diagonals of the matrix $X$, it follows that every
$x_{ij}$ occurs exactly once on the main diagonal of one matrix $M\in\mathcal{M}$.
\end{proof}
This allows us to write $f_{ij}=\rho^{-1}\left(i,j\right)$. Removing
the function $f_{55}=\det X$ (which is constant on $SL_{5}$), we
get a set of 24 regular functions on $SL_{5}$. Denote this set by
$\mathcal{B=}\left\{ f_{ij}\right\} _{i,j=1}^{5}]\setminus\{f_{5,5}\}$.
\begin{prop}
For every orientable triple $T$ the set $\mathcal{B}$ is algebraically
independent.\label{prop:AlgIndp}
\end{prop}
This can be verified by checking that the gradients of $\left\{ f_{ij}\right\} _{i,j=1}^{5}$
form a linearly independent set. 
\qed

\ 
Let $T$ be an orientable Belavin--Drinfeld triple, and let $\left\{ \cdot,\cdot\right\} _{T}$
denote the Sklyanin bracket associated with the triple $T$.
\begin{prop}
For every orientable triple $T$ the set $\mathcal{B}$ is log canonical
with respect to $\left\{ \cdot,\cdot\right\} _{T}$.\label{prop:LogCan}
\end{prop}
To show that we just compute $\frac{\left\{ f_{ij},f_{k\ell}\right\} }{f_{ij}\cdot f_{k\ell}}$
for every pair of functions $f_{ij},f_{k\ell}\in\mathcal{B}$ and
get a constant.
\qed

for every pair of functions $\phi_{i},\phi_{j}\in\mathcal{B}$ we
can compute the coefficient $\omega_{ij}=\frac{\left\{ \phi_{i},\phi_{j}\right\} }{\phi_{i}\phi_{j}}$.
Let $\Omega=\left(\omega_{ij}\right)$ be the Poisson coefficient
matrix. According to the Proposition \ref{prop:CompPsnBrkt}, a cluster
structure in $\mathcal{O}\left(SL_{5}\right)$ with a seed $\left(\mathcal{B},B\right)$,
is compatible with the Poisson structure $\left\{ \cdot,\cdot\right\} $
if $B$ satisfies $\Omega B=\left(D0\right)$ where $D$ is a diagonal
matrix. In all the cases described in this paper, $\Omega$ is a $24\times24$
matrix of rank $24$, so we simply compute $B=\Omega^{-1}$.

\subsection{Recovering the cluster structure for the orientable case}

Motivated by \cite{gekhtman2013exotic}, we look at the functions
that are determinants of the matrices $M\in\mathcal{M}$: let $S$=$\left\{ f_{5m},f_{k5}|1\leq m,k\leq4,\text{ and }m\notin\Gamma_{1},k\notin\Gamma_{2}\right\} $.
Let $\bar{S}=\left\{ \det M|M\in\mathcal{M}\setminus\left\{ X,Y\right\} \right\} $
be the set of functions $F\left(X,Y\right)$on the double that correspond
to the elements of $S$.
\begin{prop}
1. The functions $F\left(X,Y\right)$ of $\bar{S}$ are semi-invariant
of the left and right action of $D_{-}=\exp\mathfrak{d_{-}}$.

2. The functions $f\left(X\right)$ of $S$ are log canonical with
$x_{ij}$ for every $i,j\in\left[1\ldots5\right]$.\label{prop:SemiInv}
\end{prop}

Setting the set $S$ as the set of stable variables, and $\tilde{B}$
as the suitable $\left(24-\left|S\right|\right)\times24$ submatrix
of $B$ produces a seed $\left(\mathcal{B},\tilde{B}\right)$. By
its definition, the set $S$ has $2\left|\Delta\setminus\Gamma_{1}\right|=2k_{T}$
functions. One can also see that the set $\bar{S}$ has the determinants
of all matrices in $\mathcal{M}$ except $X$ and $Y$. Since the
set $\mathcal{M}$ has $\left|\mathcal{M}\right|=2k_{T}+2$
matrices, this means $\left|\bar{S}\right|=2k_{T}$. Both ways fit
assertion \ref{Conj:NumStbVar} of Conjecture \ref{Conj:GSV-BD-CS}
about the number of stable variables. For reasons that will be explained
below, we will sometimes extend $\mathcal{C}\left(\mathcal{B},\tilde{B}\right)$
from a cluster structure in $\mathcal{O}\left(SL_{5}\right)$ to
one in $\mathcal{O}\left(\mathrm{Mat}_{5}\right)$, adding $\det X$ as a stable
variable, and the appropriate column to $\tilde{B}$. We will use
this form when describing the quivers of the exotic cluster structures. 

Recall that the standard cluster structure
on $SL_{5}$ is the one that corresponds to the trivial Belavin--Drinfeld
triple $\Gamma_{1}=\Gamma_{2}=\emptyset.$ The
quiver of the standard cluster structure is shown in Figure \ref{fig:StdQvr}
(circles represent mutable variables and squares represent
stable variables). The vertex in the $i$-th row and $j$-th column
corresponds to the cluster variable $f_{ij}$. Note that in the standard
case all functions of the form $f_{5j}$ or $f_{i5}$ are stable variables. 

\begin{figure}
\begin{centering}
\includegraphics[scale=0.45]{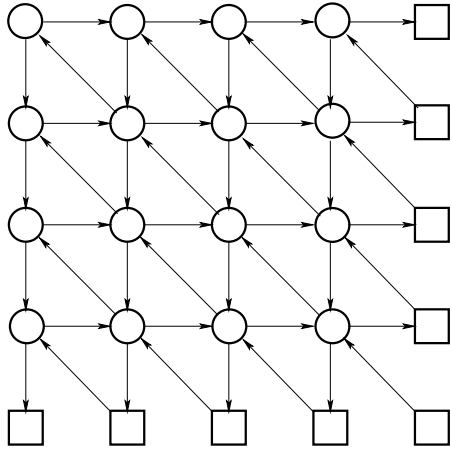}
\par\end{centering}

\caption{The quiver of the standard cluster structure}

\centering{}\label{fig:StdQvr}
\end{figure}

All the exotic quivers have similar form, with minor changes. For
a triple $T=\left(\Gamma_{1},\Gamma_{2},\gamma\right)$, let $Q_{T}$
be the quiver obtained from the standard one through the following
operations:
\begin{itemize}
\item For every $i\in\Gamma_{1}$ the stable variable $f_{5i}$ becomes
a cluster variable. Similarly, for every $j\in\Gamma_{2}$ the stable
variable $f_{j5}$ becomes a cluster variable.
\item For every new cluster variable $f_{5i}$ (that was stable in the standard
quiver), arrows are added from $f_{5i}$ to $f_{1,\gamma\left(i\right)}$,
from $f_{5i}$ to $f_{5,i+1}$ and from $f_{1,\gamma\left(i\right)+1}$
to $f_{5i}$.
\item For every new cluster variable $f_{j5}$ (that was stable in the standard
quiver), arrows are added from $f_{j5}$ to $f_{\gamma^{-1}\left(j\right),1}$,
from $f_{j5}$ to $f_{j+1,5}$ and from $f_{\gamma^{-1}\left(j\right)+1,1}$
to $f_{j5}$. 
\end{itemize}
As an example the quiver $T=\left(\left\{ 1\right\} ,\left\{ 2\right\} ,\gamma:1\mapsto2\right)$
is shown in Figure \ref{fig:Qvr1to2}. 

\begin{center}
\begin{figure}[h]
\begin{centering}
\includegraphics[scale=0.45]{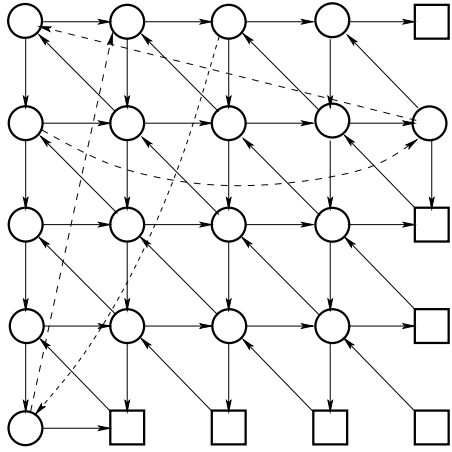}\caption{The quiver of the cluster structure $\left\{ 1\right\} \mapsto\left\{ 2\right\} $}

\par\end{centering}

\centering{}\label{fig:Qvr1to2} 
\end{figure}

\par\end{center}
\begin{rem}
If $T$ is a orientable triple with $\left|\Gamma_{1}\right|>1$,
the operations above can be applied independently for each $i\in\Gamma_{1}$
and $j\in\Gamma_{2}$. For example, the quiver of $T=\left(\left\{ 1,2\right\} ,\left\{ 2,3\right\} ,\gamma:i\mapsto i+1\right)$
can be viewed as a superposition of the quiver of $T=\left(\left\{ 1\right\} ,\left\{ 2\right\} ,\gamma:i\mapsto i+1\right)$
and the quiver of $T=\left(\left\{ 2\right\} ,\left\{ 3\right\} ,\gamma:i\mapsto i+1\right)$. \end{rem}
\begin{prop}
For every orientable triple $T$ the quiver $Q_{T}$ describes a cluster
structure on $SL_{5}$ that is compatible with the Sklyanin bracket
$\left\{ \cdot,\cdot\right\} _{T}$ associated with the triple $T$.\label{prop:QvrCmp}
\end{prop}
Note that the quiver of the standard structure is planar. The quivers
of all the exotic structures are non-planar, since there are edges
connecting $f_{i5}$ and $f_{i'1}$ or $f_{5j}$ and $f_{1j'}$. However,
any orientable exotic quiver can be embedded on the torus such that
there are no crossing edges. Figure \ref{fig:qvr1to2Trs} illustrates
such an embedding for the quiver of the cluster structure $\left\{ 1\right\} \mapsto\left\{ 2\right\} $.
We identify opposite edges of the dashed square oriented as indicated
by the arrows.

\begin{figure}
\begin{centering}
\includegraphics[scale=0.4]{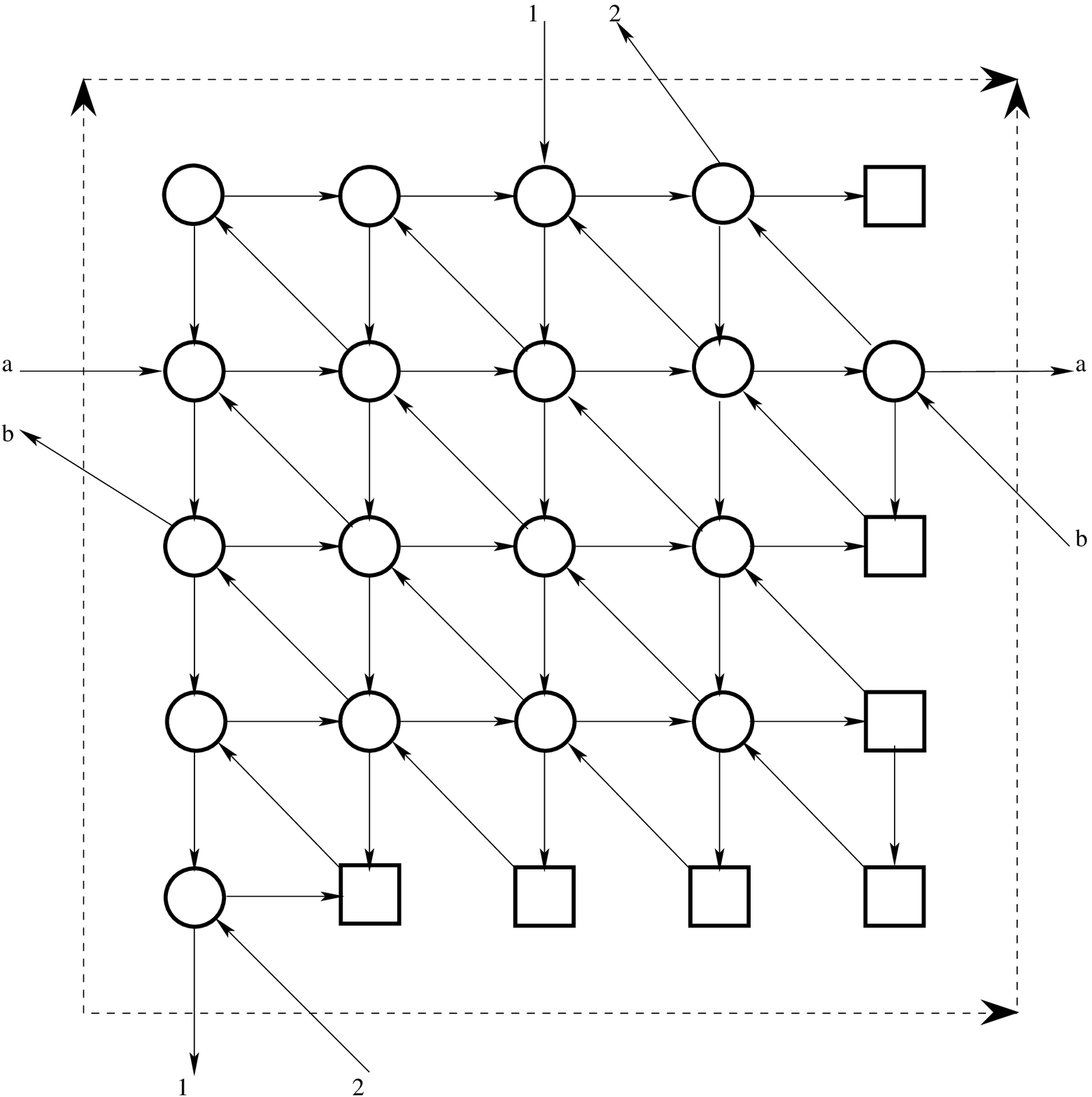}
\par\end{centering}

\caption{Embedding the quiver $1\mapsto2$ on the torus}
\label{fig:qvr1to2Trs}
\end{figure}

Starting with the seed $\left(\mathcal{B},\tilde{B}\right)$, we can
mutate in direction $k$ using the exchange relation (\ref{eq:ExRltn}). 
\begin{prop}
For every mutable cluster variable $x_{k}=f_{ij}$, the variable $x_{k}'=f_{ij}'$
defined by (\ref{eq:ExRltn}) is a regular function on $SL_{5}$.\label{prop:AdjClstrIsReg}\end{prop}
\begin{proof}
This can be verified directly by computing all the adjacent clusters
by the exchange relation (\ref{eq:ExRltn}). \end{proof}
\begin{prop}
$\overline{\mathcal{A}}\left(\mathcal{B},\tilde{B}\right)=\mathcal{O}\left(SL_{5}\right)$.\label{prop:UpCA=00003DO(SL)}\end{prop}
\begin{proof}
We use Proposition \ref{prop:ACNatIsoO(V)}. Since $SL_{5}$ is not
a Zariski open subset of $\mathbb{C}^{25}$, we extend it to $\mathrm{Mat}_{5}$
and extend the cluster structure $\mathcal{C}\left(\mathcal{B},\tilde{B}\right)$
to a cluster structure in $\mathrm{Mat}_{5}$ by adding the function $\det X$
as a stable variable. It is not hard to obtain the extra column on
the right of $\tilde{B}$ using the homogeneity of the exchange relations
and satisfying $\Omega\tilde{B}=I$. Conditions \ref{prop:ACNIOVCond1},
\ref{prop:ACNIOVCond2} and \ref{prop:ACNIOVCond4} of Proposition
\ref{prop:ACNatIsoO(V)} are clearly true, and condition \ref{prop:ACNIOVCond3}
holds by Proposition \ref{prop:AdjClstrIsReg}. The ring of regular
functions on $\mathrm{Mat}_{5}$ is generated by the matrix entries $x_{ij}$.
By Theorem 3.21 in \cite{GSV}, condition \ref{prop:ACNIOVCond1}
implies that the upper cluster algebra coincides with the intersection
of rings of Laurent polynomials in cluster variables taken over the
initial cluster and all its adjacent clusters. So it suffices to check
that every matrix entry can be expressed as a Laurent polynomial in
the variables of each of these clusters. This is verified by direct
computation with Maple. 
\end{proof}
\

\section{The non-orientable case\label{sec:NonOrientable}}

As mentioned previously, the non-orientable case 
$\Gamma_{1}=\left\{ \alpha_{1},\alpha_{2}\right\} ,\Gamma_{2}=\left\{ \alpha_{3},\alpha_{4}\right\} ,\gamma:\alpha_{i}\to\alpha_{5-i}$
is somewhat different. The main difference is the structure of the functions $f_{ij}$:
start with the same construction described
in section \ref{sec:Exotic-cluster-structures}. This gives the set
of matrices $\mathcal{M}=\left\{ M_{i}\right\} _{i=1}^{6}$ where
\begin{eqnarray*}
M_{1} & = & X,\quad M_{2}=Y,\\
M_{3} & = & \left[\begin{array}{cccccccc}
y_{{12}} & y_{{13}} & y_{{14}} & y{}_{{15}} & 0 & 0 & 0 & 0\\
y_{{22}} & y_{{23}} & y_{{24}} & y_{{25}} & 0 & 0 & 0 & 0\\
y_{{32}} & y_{{33}} & y_{{34}} & y_{{35}} & 0 & 0 & 0 & 0\\
y_{{42}} & y_{{43}} & y_{{44}} & y_{{45}} & x_{{11}} & x_{{12}} & x_{{13}} & x_{{14}}\\
y_{{52}} & y_{{53}} & y_{{54}} & y_{{55}} & x_{{21}} & x_{{22}} & x_{{23}} & x_{{24}}\\
0 & 0 & 0 & 0 & x_{{31}} & x_{{32}} & x_{{33}} & x_{{34}}\\
0 & 0 & 0 & 0 & x_{{41}} & x_{{42}} & x_{{43}} & x_{{44}}\\
0 & 0 & 0 & 0 & x_{{51}} & x_{{52}} & x_{{53}} & x_{{54}}
\end{array}\right],\\
M_{4} & = & \left[\begin{array}{cccccc}
y_{{13}} & y_{{14}} & y{}_{{15}} & 0 & 0 & 0\\
y_{{23}} & y_{{24}} & y_{{25}} & 0 & 0 & 0\\
y_{{33}} & y_{{34}} & y_{{35}} & x_{{21}} & x_{{22}} & x_{{23}}\\
y_{{43}} & y_{{44}} & y_{{45}} & x_{{31}} & x_{{32}} & x_{{33}}\\
0 & 0 & 0 & x_{{41}} & x_{{42}} & x_{{43}}\\
0 & 0 & 0 & x_{{51}} & x_{{52}} & x_{{53}}
\end{array}\right],\\
M_{5} & = & \left[\begin{array}{cccc}
x_{{41}} & x_{{42}} & x_{{43}} & 0\\
x_{{51}} & x_{{52}} & x_{{53}} & 0\\
0 & y_{{13}} & y_{{14}} & y{}_{{15}}\\
0 & y_{{23}} & y_{{24}} & y_{{25}}
\end{array}\right],\quad M_{6}=\left[\begin{array}{cc}
x_{{51}} & x_{{52}}\\
y_{{14}} & y{}_{{15}}
\end{array}\right].
\end{eqnarray*}
Define now four other matrices as follows:

\
\begin{eqnarray*}
M_{1}' & = & \left[0\right],\quad M_{2}'=\left[0\right],\\
M_{3}' & = & \left[\begin{array}{cccccccc}
0 & 0 & 0 & 0 & x_{{11}} & x_{{12}} & x_{{13}} & x_{{14}}\\
0 & 0 & 0 & 0 & x_{{21}} & x_{{22}} & x_{{23}} & x_{{24}}\\
y_{{12}} & y_{{13}} & y_{{14}} & y_{{15}} & 0 & 0 & 0 & 0\\
y_{{22}} & y_{{23}} & y_{{24}} & y_{{25}} & 0 & 0 & 0 & 0\\
y_{{42}} & y_{{43}} & y_{{44}} & y_{{45}} & 0 & 0 & 0 & 0\\
y_{{52}} & y_{{53}} & y_{{54}} & y_{{55}} & 0 & 0 & 0 & 0\\
0 & 0 & 0 & 0 & x_{{41}} & x_{{42}} & x_{{43}} & x_{{44}}\\
0 & 0 & 0 & 0 & x_{{51}} & x_{{52}} & x_{{53}} & x_{{54}}
\end{array}\right],\\
M_{4}' & = & \left[\begin{array}{cccccc}
0 & 0 & 0 & x_{{11}} & x_{{12}} & x_{{13}}\\
y_{{13}} & y_{{14}} & y{}_{{15}} & 0 & 0 & 0\\
y_{{23}} & y_{{24}} & y_{{25}} & 0 & 0 & 0\\
y_{{53}} & y_{{54}} & y_{{55}} & 0 & 0 & 0\\
0 & 0 & 0 & x_{{41}} & x_{{42}} & x_{{43}}\\
0 & 0 & 0 & x_{{51}} & x_{{52}} & x_{{53}}
\end{array}\right],\\
M_{5}' & = & \left[\begin{array}{cccc}
0 & 0 & x_{{42}} & x_{{43}}\\
0 & 0 & x_{{52}} & x_{{53}}\\
y_{{13}} & y_{{14}} & 0 & 0\\
y_{{23}} & y_{{24}} & 0 & 0
\end{array}\right],\quad M_{6}'=\left[\begin{array}{cc}
0 & x_{{53}}\\
y_{{13}} & 0
\end{array}\right].
\end{eqnarray*}
 Now for every $j\in\left[6\right]$, take all the functions $\det\left(M_{j}\right)_{\left[1\ldots r\right]}^{\left[1\ldots r\right]}+\left(-1\right)^{j}\det\left(M_{j}'\right)_{\left[1\ldots r\right]}^{\left[1\ldots r\right]}$.
Each of these function is labeled $f_{ij}$ when $\left(M_{j}\right)_{r,r}=x_{ij}$
(i.e., the map $\rho$ is defined in the same way, regarding the matrix
$M_{j}$ in the definition of $f$. This also guarantees that Proposition
\ref{prop:RhoBjctv} still holds, as it
uses the matrices as constructed for the orientable cases).

\
The set $\mathcal{B}=\left\{ f_{ij}\right\} $ is then log canonical
with respect to the Sklyanin bracket associated with the triple $\Gamma_{1}=\left\{ \alpha_{1},\alpha_{2}\right\} ,\Gamma_{2}=\left\{ \alpha_{3},\alpha_{4}\right\} ,\gamma:\alpha_{i}\to\alpha_{5-i}$
. Proceeding as described in the orientable cases one can verify that
Propositions \ref{prop:AlgIndp} - \ref{prop:QvrCmp} hold here as
well.

The quiver of this cluster structure is almost the same as in the
orientable case. The only difference is that there are no edges between
the vertices $f_{35}$ and $f_{45}$ and between the vertices $f_{51}$
and $f_{52}$ - see Figure \ref{fig:Qvr12to43}.
\begin{figure}[h]
\
\centering{}\includegraphics[scale=0.45]{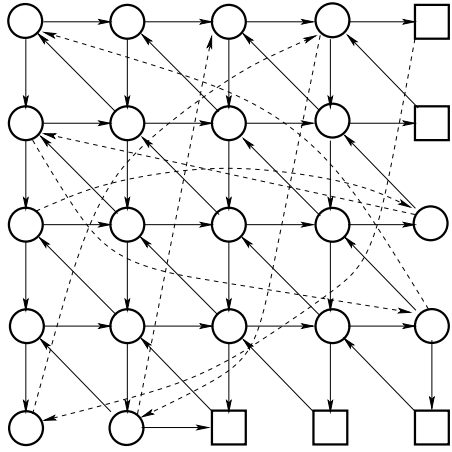}\caption{The quiver of the cluster structure $1\mapsto4,2\mapsto3$}
\label{fig:Qvr12to43}\
\end{figure}
 This quiver can not be embedded on the torus with no crossing edges.
However, it can be embedded on the projective plane. Figure \ref{fig:qvr12to43PP}
shows the quiver on the universal cover of the projective plane (identifying
opposite edges of the dashed square oriented as indicated by the arrows).
This justify the terminology: the quivers of the orientable triples
can be embedded on the torus, which is an orientable surface, while
this quiver is embedded on a non orientable surface .

\begin{figure}
\
\begin{centering}
\includegraphics[scale=0.45]{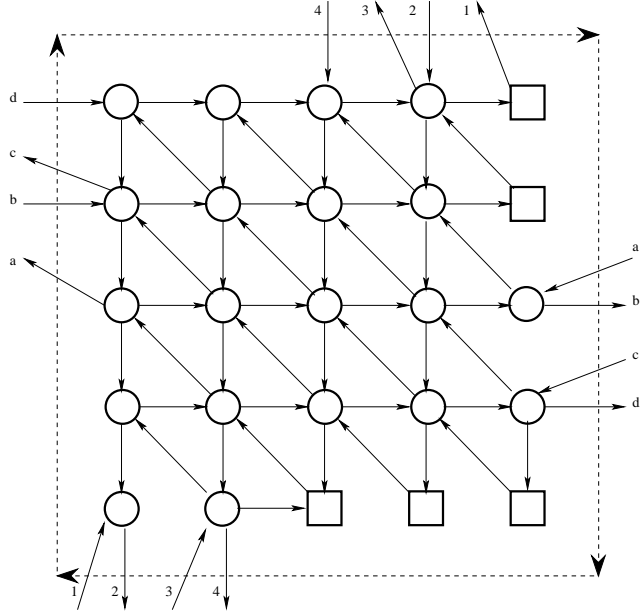}
\par\end{centering}

\
\caption{Embedding the quiver $1\mapsto4,2\mapsto3$
on the projective plane}

\label{fig:qvr12to43PP}
\end{figure}

The removal of the two edges $\left(f_{34},f_{45}\right)$ and $\left(f_{51},f_{52}\right)$
can be explained as an attempt to preserve the structure of the quiver
on the projective plane: when we identify the
opposite edges of the dashed square in Figure \ref{fig:qvr12to43PP},
and look at the vertices $f_{51}$ and $f_{52}$ we expect an edge
$f_{51}\longrightarrow f_{52}$, as there is an edge directed to the
right between any two horizontally adjacent vertices. But if we would
put their copies $f_{51}'$ above $f_{14}$ and $f_{52}'$ above $f_{13}$
the edge should be directed from $f_{52}'$ to $f_{51}'$ , because
now $f_{51}'$ is on the right side of $f_{52}'$. The same holds
for the edge between $f_{35}$ and $f_{45}$. This contradiction is
settled by  removing these two edges.

\section{Toric action}

Assertion \ref{conj:global-toric} of Conjecture \ref{Conj:GSV-BD-CS}
can be interpreted as follows: 

For any $H\in\mathcal{H}$ and any weight $\omega\in\mathfrak{h}^{*}$
put $H^{\omega}=e^{\omega\left(h\right)}$ where $H=\exp h$. Let
$\left(\tilde{\mathbf{x}},\tilde{B}\right)$ be a seed in $\mathcal{C}_{T}$,
and $y_{i}=\varphi\left(x_{i}\right)$ for $i\in\left[n+m\right]$.
Then \ref{conj:global-toric} is equivalent to the following:
\begin{enumerate}
\item for any $H_{1},H_{2}\in\mathcal{H}_{T}$ and any $x\in\mathcal{G}$,
\[
y_i\left(H_{1}XH_{2}\right)=H_{1}^{\eta_{i}}H_{2}^{\zeta_{i}}y_i\left(X\right)
\]
 for some weights $\eta_{i},\zeta_{i}\in\mathfrak{h}_{T}^{*}$ ($i\in\left[n+m\right]$);\label{RmkGTACond1}
\item $\spn\left\{ \eta_{i}\right\} _{i=1}^{\dim\mathcal{G}}=\spn\left\{ \zeta_{i}\right\} _{i=1}^{\dim\mathcal{G}}=\mathfrak{h}_{T}^{*}$
;\label{RmkGTACond2}
\item for every $i\in\left[\dim\mathcal{G}-2k_{T}\right]$, 
\[
\sum_{j=1}^{\dim\mathcal{G}}b_{i,j}\eta_{j}=\sum_{j=1}^{\dim\mathcal{G}}b_{i,j}\zeta_{j}=0.
\]
 \label{RmkGTACond3}
\end{enumerate}
\
To verify assertion \ref{conj:global-toric} of Conjecture \ref{Conj:GSV-BD-CS}
, we parametrize the left and the right action of  $\mathcal{H}_{T}$
for every triple $T$ by diagonal matrices as shown in Table \ref{Tbl:prmtrz}.

\begin{table}[H]
\begin{centering}
\begin{tabular}{|c|l|c|}
\hline 
\
B--D class\
 & \
Parametrization\
 & \
dimension of $\mathcal{H}_{T}$\
\tabularnewline
\hline 
\
2.\
 & $\diag\left(r,s,s^{2}r^{-1},ts^{-3},t^{-1}\right)$ & \
$3$\
\tabularnewline
\hline 
\
3.\
 & \
$\diag\left(r,sr^{-1},rts^{-1},tr^{-1},t^{-2}\right)$\
 & \
$3$\
\tabularnewline
\hline 
\
4.\
 & \
$\diag\left(r,s,r^{-2}t^{2},rrs^{-1}t^{-1},t^{-1}\right)$\
 & \
$3$\
\tabularnewline
\hline 
\
5.\
 & \
$\diag\left(r^{3}t,r^{-2}st^{-1},r^{-1},ts^{-1},t^{-1}\right)$\
 & \
$3$\
\tabularnewline
\hline 
\
6.\
 & \
$\diag\left(r,s,r^{-1}s^{2},r^{-2}s^{3},r^{2}s^{-}6\right)$\
 & \
$2$\
\tabularnewline
\hline 
\
7.\
 & \
$\diag\left(r^{2},r^{-1}s^{2},rs^{-1}sr^{-2},s^{-2}\right)$\
 & \
$2$\
\tabularnewline
\hline 
\
8.\
 & \
$\diag\left(r^{2}s^{4},rs^{-1},s^{-6},r^{-1}s^{-1},r^{-2}s^{4}\right)$\
 & \
$2$\
\tabularnewline
\hline 
\
9.\
 & \
$\diag\left(r,s,1,s^{-1},r^{-1}\right)$\
 & \
$2$\
\tabularnewline
\hline 
\
10.\
 & \
$\diag\left(r^{3},s^{3},r^{2}s^{-4},r^{-1}s^{-1},r^{-4}s^{2}\right)$\
 & \
$2$\
\tabularnewline
\hline 
\
11.\
 & \
$\diag\left(r,s,1,s^{-1},r^{-1}\right)$\
 & \
$2$\
\tabularnewline
\hline 
\
12.\
 & \
$\diag\left(r^{2},r,1,r^{-1},r^{-2}\right)$\
 & \
$1$\
\tabularnewline
\hline 
\
13.\
 & \
$\diag\left(r^{2},r,1,r^{-1},r^{-2}\right)$\
 & \
$1$\
\tabularnewline
\hline 
\end{tabular}
\par\end{centering}

\caption{Parametrization of left and right action of $\mathcal{H}_{T}$}

\label{Tbl:prmtrz}
\end{table}

As an example we look at case 2.$\left\{ 1\right\} \mapsto\left\{ 2\right\} $:
the torus $\mathcal{H}_{1\mapsto2}$ has dimension $3$. We parametrize
the left and the right action of $\mathcal{H}_{1\mapsto2}$ by $\diag(r,s,s^{2}r^{-1},ts^{-3},t^{-1})$
and $\diag(u,v,v^{2}u^{-1},wv^{-3},w^{-1})$, respectively. Then condition
\ref{RmkGTACond1} above holds with $3$-dimensional vectors $\eta_{i},\zeta_{i}$
given by \vskip .5cm

\begin{center}
\begin{tabular}{llll}
$\eta_{1}=(1,0,0)$ & $\eta_{2}=(1,1,0)$ & $\eta_{3}=(1,0,0)$ & $\eta_{4}=(0,0,-1)$\tabularnewline
$\eta_{5}=(1,2,0)$ & $\eta_{6}=(1,0,0)$ & $\eta_{7}=(1,1,0)$ & $\eta_{8}=(1,0,-1)$\tabularnewline
$\eta_{9}=(-1,2,0)$ & $\eta_{10}=(0,4,0)$ & $\eta_{11}=(1,1,0)$ & $\eta_{12}=(0,3,0)$\tabularnewline
$\eta_{13}=(1,1,-1)$ & $\eta_{14}=(0,-3,1)$ & $\eta_{15}=(-1,-1,1)$ & $\eta_{16}=(0,1,1)$\tabularnewline
$\eta_{17}=(0,3,0)$ & $\eta_{18}=(0,0,1)$ & $\eta_{19}=(0,3,-1)$ & $\eta_{20}=(0,-3,0)$\tabularnewline
$\eta_{21}=(-1,-1,0)$ & $\eta_{22}=(1,0,0)$ & $\eta_{23}=(0,1,0)$ & $\eta_{24}=(0,0,1)$ \tabularnewline
\end{tabular} 
\par\end{center}

\vspace*{\smallskipamount}

\begin{center}
\begin{tabular}{llll}
$\zeta_{1}=(0,-3,1)$  & $\zeta_{2}=(0,-3,0)$  & $\zeta_{3}=(1,0,0)$  & $\zeta_{4}=(1,0,0)$\tabularnewline
$\zeta_{5}=(1,-3,0)$  & $\zeta_{6}=(0,1,0)$  & $\zeta_{7}=(1,1,0)$  & $\zeta_{8}=(0,2,0)$\tabularnewline
$\zeta_{9}=(1,0,0)$  & $\zeta_{10}=(1,-2,0)$  & $\zeta_{11}=(-1,3,0)$  & $\zeta_{12}=(0,3,0)$\tabularnewline
$\zeta_{13}=(0,-1,1)$  & $\zeta_{14}=(1,0,0)$  & $\zeta_{15}=(1,1,0)$  & $\zeta_{16}=(0,0,0)$\tabularnewline
$\zeta_{17}=(-1,0,1)$  & $\zeta_{18}=(0,0,1)$  & $\zeta_{19}=(0,-1,0)$  & $\zeta_{20}=(1,1,0)$\tabularnewline
$\zeta_{21}=(0,3,0)$  & $\zeta_{22}=(0,0,-1)$  & $\zeta_{23}=(0,-3,1)$  & $\zeta_{24}=(-1,0,0)$ \tabularnewline
\end{tabular}
\par\end{center}

\vspace*{\medskipamount}

The above conditions \ref{RmkGTACond2} and \ref{RmkGTACond3} can
now be verified via direct computation.

\section{Remark about computations}

All the computations were done with Maple. Although a computer of
128Gb RAM was used, there were two cases where the computations required
more than that, and therefore failed. The first is for 
$\Gamma_{1}=\left\{ \alpha_{1},\alpha_{3}\right\} ,\Gamma_{2}=\left\{ \alpha_{1},\alpha_{4}\right\} ,\gamma:\alpha_{i}\to\alpha_{i+3\pmod5}$. 
Here only one Poisson coefficient could not be computed -- the coefficient
$\omega=\frac{\left\{ f_{4,3},f_{5,4}\right\} }{f_{4,3}\cdot f_{5,4}}$.
The value of $\omega$ was evaluated at a random point of $SL_{5}$,
and this value matches the structure described in Section \ref{sec:Exotic-cluster-structures}.
The problem was bigger in the case $\Gamma_{1}=\left\{ \alpha_{1},\alpha_{2},\alpha_{4}\right\} ,\Gamma_{2}=\left\{ \alpha_{1},\alpha_{3},\alpha_{4}\right\} ,\gamma:\alpha_{i}\to\alpha_{i+2\pmod5}$.
Here there were $28$ pairs of functions for which the coefficient $\omega$
could not be computed. Again, for each of these pairs $\omega$ was
evaluated at a random point of $SL_{5}$ . These values all match the construction
described above.

\section*{Acknowledgments}

The author was supported by ISF grant \#162/12. Parts of this paper
were written during my stay at MSRI (Cluster Algebras program, August-
December, 2012). I would like to thank this institution for the warm
hospitality and excellent working conditions. 
I am grateful to Michael Gekhtman for his enlightening comments at certain points.
Special thanks to Alek Vainshtein 
for his support and encouragement, as well as his mathematical,
technical and editorial advices.

\pagebreak{}

\bibliographystyle{abbrv}
\bibliography{EXCLSL5}

\def\cprime{$'$}
\begin{thebibliography}{10}

\bibitem{BDSolCYBE}
A.~A. Belavin and V.~G. Drinfel{\cprime}d.
\newblock Solutions of the classical {Y}ang-{B}axter equation for simple {L}ie
  algebras.
\newblock {\em Funktsional. Anal. i Prilozhen.}, 16(3):1--29, 96, 1982.

\bibitem{BFZ}
A.~Berenstein, S.~Fomin, and A.~Zelevinsky.
\newblock Cluster algebras. {III}. {U}pper bounds and double {B}ruhat cells.
\newblock {\em Duke Math. J.}, 126(1):1--52, 2005.

\bibitem{chriprsly}
V.~Chari and A.~Pressley.
\newblock {\em A guide to quantum groups}.
\newblock Cambridge University Press, Cambridge, 1994.

\bibitem{FZ1}
S.~Fomin and A.~Zelevinsky.
\newblock Cluster algebras. {I}. {F}oundations.
\newblock {\em J. Amer. Math. Soc.}, 15(2):497--529 (electronic), 2002.

\bibitem{FZ2}
S.~Fomin and A.~Zelevinsky.
\newblock Cluster algebras. {II}. {F}inite type classification.
\newblock {\em Invent. Math.}, 154(1):63--121, 2003.

\bibitem{GSV1}
M.~Gekhtman, M.~Shapiro, and A.~Vainshtein.
\newblock Cluster algebras and {P}oisson geometry.
\newblock {\em Mosc. Math. J.}, 3(3):899--934, 1199, 2003.
\newblock \{Dedicated to Vladimir Igorevich Arnold on the occasion of his 65th
  birthday\}.

\bibitem{GSV}
M.~Gekhtman, M.~Shapiro, and A.~Vainshtein.
\newblock {\em Cluster algebras and {P}oisson geometry}, volume 167 of {\em
  Mathematical Surveys and Monographs}.
\newblock American Mathematical Society, Providence, RI, 2010.

\bibitem{gekhtman2012cluster}
M.~Gekhtman, M.~Shapiro, and A.~Vainshtein.
\newblock Cluster structures on simple complex {L}ie groups and
  {B}elavin-{D}rinfeld classification.
\newblock {\em Mosc. Math. J.}, 12(2):293--312, 460, 2012.

\bibitem{gekhtman2013exotic}
M.~Gekhtman, M.~Shapiro, and A.~Vainshtein.
\newblock Exotic cluster structures on ${SL}_n$: the cremmer-gervais case.
\newblock {\em arXiv preprint arXiv:1307.1020}, 2013.

\bibitem{reyman1994group}
A.~G. Reyman and M.~A. Semenov-Tian-Shansky.
\newblock Group-theoretical methods in the theory of finite-dimensional
  integrable systems.
\newblock {\em Dynamical Systems VII, Editors V.I. Arnold and S.P. Novikov,
  Encyclopaedia of Mathematical Sciences}, 16:116--225, 1994.

\bibitem{Scott}
J.~S. Scott.
\newblock Grassmannians and cluster algebras.
\newblock {\em Proc. London Math. Soc. (3)}, 92(2):345--380, 2006.

\end{thebibliography}

\end{document}